\documentclass[a4paper,12pt,reqno]{amsart}
\usepackage{latexsym}
\usepackage{amssymb} 
\usepackage{mathrsfs}
\usepackage{amsmath}
\usepackage{latexsym}
\usepackage{delarray}
\usepackage{amssymb,amsmath,amsfonts,amsthm,mathrsfs}

\usepackage{hyperref}
\renewcommand\eqref[1]{(\ref{#1})} 

 \newtheorem{thm}{Theorem}[section]
 \newtheorem{cor}[thm]{Corollary}
 \newtheorem{lem}[thm]{Lemma}
 \newtheorem{prop}[thm]{Proposition}
 \theoremstyle{definition}
 \newtheorem{defn}[thm]{Definition}
 \theoremstyle{remark}
 \newtheorem{rem}[thm]{Remark}
 
 \numberwithin{equation}{section}
\newcommand{\half}{\frac{1}{2}}

\newcommand{\ene}{\mathbb{N}}
\newcommand{\noi}{\noindent}

\newcommand{\ar}{\mathbb{R}}
\newcommand{\ce}{\mathbb{C}}
\newcommand{\zet}{\mathbb{Z}}

\newcommand{\ern}{{\mathbb{R}}^n}
\newcommand{\arn}{{\mathbb{R}}^n}

\newcommand{\bi}{\begin{itemize}}
\newcommand{\ei}{\end{itemize}}
\newcommand{\be}{\begin{enumerate}}
\newcommand{\ee}{\end{enumerate}}
\newcommand{\beq}{\begin{equation}}
\newcommand{\eq}{\end{equation}}

\newcommand{\lap}{\mathcal{L}_G}
\newcommand{\cdxi}{\ce^{d_{\xi}\times d_{\xi}}}

\newcommand{\ernt}{\mathbb{R}^n\times\mathbb{R}^n}

\def\p#1{{\left({#1}\right)}}

\def\jp#1{{\left\langle{#1}\right\rangle}}

\def\Rep{{{\rm Rep}}}
\def\Op{{{\rm Op}}}

\def\D{\mathbb D}
\def\rhodiff{\mbox{$\triangle\!\!\!\!\ast\,$}}

\DeclareMathOperator{\Tr}{Tr}

\DeclareMathOperator{\rank}{rank}

\DeclareMathOperator*{\esssup}{ess\,sup}

\def\Gh{{\widehat{G}}}

\def\HS{{\mathtt{HS}}}
\def\Rn{{{\mathbb R}^n}}
\def\Tn{{{\mathbb T}^n}}
\def\Zn{{{\mathbb Z}^n}}

\def\SU2{{{\rm SU(2)}}}
\def\SO3{{{\rm SO(3)}}}
\def\lapsu2{{{\mathcal L}_\SU2}}

\def\rhodiff{\mbox{$\triangle\!\!\!\!\ast\,$}}

\def\Re{{{\rm Re}\,}}
\def\Op{\text{\rm Op}}

\begin{document}

%
\title[$L^p$-bounds for pseudo-differential operators]{$L^p$-bounds for pseudo-differential operators on compact Lie groups}
\author[Julio Delgado]{Julio Delgado}


\address{Department of Mathematics\\
Imperial College London\\
180 Queen's Gate, London SW7 2AZ\\
United Kingdom}

\email{j.delgado@imperial.ac.uk}

\thanks{The first author was supported by the
Leverhulme Research Grant RPG-2014-02.
The second author was supported by
 the EPSRC Grant EP/K039407/1. No new data was collected or generated during the course of research.}

\author{Michael Ruzhansky}

\address{Department of Mathematics\\
Imperial College London\\
180 Queen's Gate, London SW7 2AZ\\
United Kingdom}

\email{m.ruzhansky@imperial.ac.uk}

\subjclass[2010]{Primary 35S05; Secondary 22E30, 47G30.}

\keywords{Compact Lie groups, pseudo-differential operators, $L^p$ bounds. }

\date{\today}

\begin{abstract}
Given a compact Lie group $G$, in this paper we establish $L^p$-bounds for pseudo-differential operators in $L^p(G)$.
  The criteria here are given in terms of the concept of matrix symbols defined on the non-commutative analogue of the phase space $G\times\widehat{G}$, where $\widehat{G}$ is the unitary dual of $G$. We obtain two different types of $L^p$ bounds: first for finite regularity symbols and second for smooth symbols. The conditions for smooth symbols are formulated using $\mathscr{S}_{\rho,\delta}^m(G)$ classes which are a suitable extension of the well known $(\rho,\delta)$ ones on the Euclidean space. The results herein extend classical $L^p$ bounds established by C. Fefferman on $\arn$.
  While Fefferman's results have immediate consequences on general manifolds for $\rho>\max\{\delta,1-\delta\}$, our results do not require the condition $\rho>1-\delta$. Moreover, one of our results also does not require $\rho>\delta$. Examples are given for the case of SU(2)$\cong\mathbb S^3$ and vector fields/sub-Laplacian operators when operators in the classes $\mathscr{S}_{0,0}^m$
  and $\mathscr{S}_{\frac12,0}^m$ naturally appear, and where conditions $\rho>\delta$ and $\rho>1-\delta$ fail, respectively.

\end{abstract}

\maketitle
\tableofcontents

\section{Introduction}
In this work we study the $L^p$ boundedness of pseudo-differential operators on compact Lie groups. The investigation of the behaviour of pseudo-differential operators of H\"ormander's class $S_{\rho,\delta}^m$ in $L^p$ is a fundamental problem in the theory of pseudo-differential operators. The fact that the class $S_{1,\delta}^0$ begets  bounded operators on $L^p$ for every $1<p<\infty$ is well known (e.g. \cite[Ch. 13]{tay:pde3}). The boundedness on $L^p(\ern)$ for all $1<p<\infty$ fails for general operators with symbols in  $S_{\rho,\delta}^{0}(\ernt)$ with $\rho<1$. Furthermore, when $m>0$ is small, for operators with symbols 
 in $S_{\rho,\delta}^{-m}(\ernt)$ with $\rho<1$ one can only get $L^p(\ern)$ boundedness for finite intervals centered at $p=2$, which is a consequence of C. Fefferman's estimates (cf. \cite{fe:lp}) and the work on multipliers of Hirschman (e.g. \cite{hi:mult}) and Wainger (cf. \cite{wai:trig}). The obstruction for the boundedness on $L^p(\ern)$ for all $1<p<\infty$ of operators in ${\rm Op}S_{\rho,\delta}^{0}(\ernt)$ with $\rho<1$ is explained in a more general setting by the works of Richard Beals \cite{be:lpn} and \cite{be:lp}. The C. Fefferman's results were extended to symbols with finite regularity by Li and Wang in \cite{wa-li:lp}. A version of $L^p$ Fefferman type bounds in the setting of $S(m,g)$ classes has been established in \cite{j:t}. 

The $L^p$ boundedness on compact groups for invariant operators (Fourier multipliers) with symbols of finite regularity has been studied in \cite{Ruzhansky-Wirth:Lp-FAA, Ruzhansky-Wirth:Lp-Z}. The case of the circle has been considered in \cite{mws:s1}.

The $(\rho,\delta)$ classes $\mathscr{S}_{\rho,\delta}^m(G)$ on compact Lie groups with $0\leq\delta\leq\rho\leq 1$ $(\delta<1)$  have been  introduced in \cite{rt:book} and then in \cite{Ruzhansky-Turunen:JFA-Garding} motivated by the study of sharp G\aa rding inequalities, and subsequently developed in \cite{Ruzhansky-Turunen-Wirth:arxiv, Ruzhansky-Wirth:Lp-FAA, Ruzhansky-Wirth:Lp-Z, Ruzhansky-Wirth:functional-calculus, vf:psecl}. 

In this paper we first establish $L^p$ bounds for finite regularity symbols by applying multiplier  results of \cite{Ruzhansky-Wirth:Lp-FAA} and \cite{Ruzhansky-Wirth:Lp-Z}. Secondly, we extend Fefferman's bounds to compact Lie groups obtaining some improvement with respect to the range of $(\rho, \delta)$ from the point of view of pseudo-differential operators on compact manifolds.  Our analysis will be based on the global 
quantization developed in \cite{rt:book} and \cite{rt:groups} as a noncommutative analogue 
of the Kohn-Nirenberg quantization of operators on $\Rn$.  The classes $\mathscr{S}_{\rho,\delta}^m(G)$ on a compact Lie group $G$ extend
 the corresponding H\"ormander ones when $G$ is viewed as a manifold. The advantage here is that we will not impose the usual restriction 
 $1-\delta\leq\rho$ when dealing with those classes on manifolds. Thus, here we will allow $\rho\leq \half$ and $\rho=\delta$.

In order to illustrate our main results we first recall the $L^{p}(\arn)$ bounds obtained by C. Fefferman (\cite{fe:lp}). In the following theorem, we denote by $\sigma(x,D)$ the pseudo-differential operator with symbol $\sigma(x,\xi)$, i.e. 
$$\sigma(x,D)f(x)=\int_{\Rn} e^{2\pi i x\cdot\xi} \sigma(x,\xi) \widehat{f}(\xi) d\xi.$$

\medskip
\noindent {\bf{Theorem A.}}\\
{
\textit{{\em(a)} Let $0\leq\delta<\rho<1$ and $\nu<n(1-\rho)/2$. Let
$\sigma=\sigma (x,\xi)\in
S_{\rho,\,\delta}^{-\nu }(\ern)$.
Then $\sigma(x,D)$ is a bounded operator from $L^p(\ern)$ to 
$L^p(\ern)$ for 
$$\left|\half-\frac{1}{p}\right|\leq \frac{\nu}{n(1-\rho)}.$$ 
{\em(b)} If $\left|\half-\frac{1}{p}\right|>\frac{\nu}{n(1-\rho)}$, then the operator
$\sigma(x,D)$ associated to the symbol
\beq\sigma(x,\xi)=\sigma_{\rho,\nu}(\xi)=e^{i|\xi|^{1-\rho} }\langle\xi\rangle^{-\nu}\in S_{\rho,\,0}^{-\nu }, \label{hlhw}\eq
is not bounded from $L^p(\ern)$ to $L^p(\ern)$.\\
{\em(c)} Let $\sigma=\sigma(x,\xi)\in S_{\rho,\,\delta}^{-n(1-\rho)/2}(\ern).$ Then 
$\sigma(x,D)$ does not have to be bounded on $L^1(\ern)$ but is bounded from the Hardy space $H^1(\ern)$ to $L^1(\ern)$.}}

\medskip
The part (a) can be deduced from (c) by complex interpolation. The part (b) corresponds to the classical counter-example due to Hardy-Littlewood-Hirschman-Wainger (cf. \cite{az:st}, \cite{hi:mult}, \cite{wai:trig}). The complex interpolation and the duality $(H^1)'=BMO$ obtained
by  Stein and Fefferman in \cite{fe:hp} reduce the proof of (c) to the estimation of $L^{\infty}-BMO$ bounds. We note that the conditions on the symbol in Theorem A restrict the choice of the parameter $\rho$ depending on the order of the symbol. Part (b) shows that part (a) is sharp with respect to the size of the interval around $p=2$. Moreover the sharpness of the choice of the value $\rho$ is explained by an estimate due to H\"ormander \cite{ho:psym}. 
In this paper, one of our main results will give an analogue of Theorem A on compact Lie groups. 
\medskip

In Section 2 we recall basic elements of the theory of pseudo-differential operators on compact Lie groups.
 In Section 3 we establish our main results on $L^p$ boundedness, we consider two types of conditions, the first ones imposing finite regularity on the symbol and the second ones for $C^{\infty}$-smooth symbols.

\medskip
To give a taste of our results we state two of our main theorems. 
Here, we rely on the global noncommutative analogue of the Kohn-Nirenberg quantization 
\eqref{EQ:A-quant} on a compact Lie group $G$ developed in \cite{rt:book, rt:groups} providing a one-to-one correspondence between matrix symbols $\sigma$ on the noncommutative phase space $G\times \widehat{G}$ and the corresponding operators $A\equiv\sigma(x,D)$ given by 
\begin{equation}\label{EQ:A-quant0}
Af(x)\equiv \sigma(x,D)f(x)=\sum\limits_{[\xi]\in \widehat{G}}d_{\xi}\Tr(\xi(x)\sigma(x,\xi)\widehat{f}(\xi)).
\end{equation}
We refer to Section \ref{SEC:Prelim} for the precise definitions of the appearing objects.

The following limited regularity result corresponds to
 Theorem \ref{thpr}. The notation ${\mathbb D}_{\xi}$ will indicate a suitable difference operator with respect to the discrete unitary dual.

\begin{thm}\label{THM1}
Let $G$ be a compact Lie group of dimension $n$, and 
let $0\leq\delta,\rho\leq 1$. Denote by $\kappa$  the smallest even integer larger than $\frac{n}{2}$.  Let $1<p<\infty$ and $\ell>\frac{n}{p}$ with $\ell\in\ene$. Let $A:C^{\infty}(G)\rightarrow\mathcal{D}'(G)$ be a linear continuous operator such that its matrix symbol $\sigma$ satisfies
\beq\label{aieq1aa}\|\partial_x^{\beta}{\mathbb D}_{\xi}^{\alpha}\sigma(x,\xi)\|_{op}\leq C_{\alpha, \beta}\langle\xi\rangle^{-m_0-\rho|\alpha|+\delta|\beta|}\eq
with $$m_0\geq\kappa(1-\rho)\left|\frac{1}{p}-\half\right|+\delta([\frac{n}{p}]+1),$$
for all multi-indices $\alpha, \beta$ with $|\alpha|\leq \kappa$, $|\beta|\leq\ell$ and for all $x\in G$ and $[\xi]\in \widehat{G}$. Then the operator $A$ is bounded from $L^{p}(G)$ to $L^p(G)$. 
\end{thm}
Here $[\frac{n}{p}]$ denotes the integer part of $\frac{n}{p}$.

For smooth symbols we will prove the following  theorem which corresponds to Theorem \ref{tci}. In particular, we will not impose
 the well known restrictions $\rho>\frac12$ or $\rho\geq 1-\delta$ when dealing with pseudo-differential operators on closed manifolds as an advantage of the 
 global calculus employed here.

\begin{thm}\label{THM2}
Let $G$ be a compact Lie group of dimension $n$. Let $0\leq  \delta<\rho<1$ and $$0\leq\nu <\frac{n(1-\rho)}{2}.$$ Let $\sigma\in \mathscr{S}_{\rho,\delta}^{-\nu}(G)$. Then $\sigma(x,D)$ extends to a bounded operator from $L^{p}(G)$ to $L^{p}(G)$ for $$\left|\frac{1}{p}-\half\right|\leq \frac{\nu}{n(1-\rho)} .$$ 
\end{thm}

Let us compare the statement of Theorem \ref{THM2} on compact Lie groups with Theorem A on $\mathbb R^{n}$. First we can observe that Theorem A readily yields the corresponding $L^{p}$-boundedness result on a general compact manifold $M$ for pseudo-differential operators with symbols in class $S^{-\nu}_{\rho,\delta}$. However, for these classes to be invariantly defined on $M$ one needs the condition $\rho\geq 1-\delta$ (see e.g. \cite{shubin:r}). Together with condition $\rho>\delta$ this implies, in particular, that $\rho>\frac12$.  Therefore, 
Theorem A implies the statement of Theorem \ref{THM2} under the additional assumption that $\rho\geq 1-\delta$ (and hence also $\rho>\frac12$), see also Remark \ref{eq4a} for the relation between operators in these symbol classes. Thus, the main point of Theorem \ref{THM2} is to establish the $L^{p}$-boundedness without this restriction. This is possible due to the global symbolic calculus available thanks to $G$ being a group. We point out that the condition $\delta\leq\rho$ has been also removed for the $L^2$-boundedness on $\arn$ by J. Hounie \cite{hou:l2}.

\section{Motivation and applications}

Let us give several examples of one type of applications and relevance of the obtained results. Let $G={\rm SU(2)}\simeq\mathbb S^{3}$ be equipped with the usual matrix multiplication of SU(2), or with the quaternionic product on $\mathbb S^{3}$. Let $X,Y,Z$ be three left-invariant vector fields, orthonormal with respect to the Killing form. Then we have the following properties, established in 
\cite{Ruzhansky-Turunen-Wirth:arxiv}:
\begin{itemize}
\item[(i)] Let ${\mathcal L}_{sub}=X^{2}+Y^{2}$ be the sub-Laplacian (hypoelliptic by H\"ormander's sum of squares theorem). Then its parametrix ${\mathcal L}_{sub}^{\sharp}$ has symbol in the symbol class $\mathscr{S}_{\frac12,0}^{-1}(G)$.
\item[(ii)] Let $\mathcal H=X^{2}+Y^{2}-Z$, it is also hypoelliptic by H\"ormander's sum of squares theorem. Then its parametrix $H^{\sharp}$ has symbol in the symbol class $\mathscr{S}_{\frac12,0}^{-1}(G)$.
\item[(iii)] The operator $Z+c$ is globally hypoelliptic if and only if ${\mathrm i}c\not\in\frac12\mathbb Z$. In this case its inverse $(Z+c)^{-1}$ exists and has symbol in the symbol class $\mathscr{S}_{0,0}^{0}(G)$.
\end{itemize}
We note that especially in the case (iii), the class $\mathscr{S}_{0,0}^{0}(G)$ is not invariantly defined in local coordinates while our global definition makes sense. For examples (i) and (ii), the class $\mathscr{S}_{\frac12,0}^{-1}(G)$ in local coordinates gives the H\"ormander class
$\mathscr{S}_{\frac12,\frac12}^{-1}(\mathbb R^{3})$. Consequently, Theorem A can not be applied since the condition $\rho>\delta$ is not satisfied in this case. Nevertheless, the results obtained in this paper apply, for example Theorem \ref{THM1} works for the class $\mathscr{S}_{0,0}^{0}(G)$, and both Theorem \ref{THM1} and Theorem \ref{THM2} for the class 
$\mathscr{S}_{\frac12,0}^{-1}(G)$. 
Thus, this can be used to derive a-priori estimates in $L^{p}$ Sobolev spaces $W^{p,s}$, for example
$$
\|f\|_{L^p(\mathbb S^3)} \le C_p 
  \|(Z+c)f\|_{W^{p,2|\frac1p-\frac12|}(\mathbb S^3)},\quad
  1<p<\infty,
$$
or
 $$\|f\|_{L^{p}(\mathbb S^3)}\leq 
C_p\|\mathcal L u\|_{W^{p,|\frac1p-\frac12|-1}(\mathbb S^3)}, \quad
  1<p<\infty,
$$ 
and
$$\|f\|_{L^{p}(\mathbb S^3)}\leq C_p \| \mathcal H u
\|_{W^{p,|\frac1p-\frac12|-1}(\mathbb S^3)},\quad
  1<p<\infty.
$$ 

\medskip
The above examples show that similarly to the introduction of the classes $S^m_{\rho,\delta}$ by H\"ormander in the analysis of hypoelliptic operators on $\Rn$, the classes $\mathscr{S}_{\rho,\delta}^{m}(G)$ also appear in the analysis of (already) Fourier multipliers on Lie groups. Moreover, if a Lie group is acting on a homogeneous manifolds $G/K$, the Fourier analysis on $G$ gives rise to the Fourier analysis on $G/K$ in terms of class I representations of $G$.

We refer to \cite{Ruzhansky-Turunen-Wirth:arxiv} and 
\cite{Ruzhansky-Wirth:Lp-Z} for other examples of the appearance of the globally defined classes 
$\mathscr{S}_{\rho,\delta}^{-\nu}(G)$ in the context of compact Lie groups, also for non-invariant operators, but let us give one explicit example here. 

Let $f\in C^\infty(G)$ be a smooth function on $G$ and let $\mathcal L$ be the Laplacian on $G$. Consider the (Schr\"odinger type) evolution problem
\begin{equation}\label{EQ:cp}
\left\{\begin{array}{rl}
i\partial_t u + f(x) (1-{\mathcal L})^{\delta/2}u & =0, \\
u(0)&= I.
\end{array}
\right.
\end{equation}
Then, modulo lower order terms, the main term of its solution operator can be seen as a pseudo-differential operator with symbol 
$$\sigma_t(x,\xi)= e^{it f(x)\jp{\xi}^\delta},$$
where $\jp{\xi}$ stands for the eigenvalue of the elliptic operator $(1-{\mathcal L})^{1/2}$ corresponding to the representation $\xi$. One can check, for example using the functional calculus from \cite{Ruzhansky-Wirth:functional-calculus}, that $\sigma \in \mathscr{S}_{\rho,\delta}^{0}(G)$, with $\rho=1-\delta$. In particular, we may have $\rho\leq \delta$, depending on the range of $\delta$ in \eqref{EQ:cp}. 

This example can be extended further if we take $\mathcal L$ in \eqref{EQ:cp} to be a sub-Laplacian, since the matrix symbol of the sub-Laplacian can be also effectively controlled, see e.g. \cite{Garetto-Ruzhansky:JDE}, wth further dependence on indices, as e.g. already in the case of $\mathbb S^3$ in (i) above.

There are other examples of problems that can be effectively treated by the global calculus rather than by localisations of the classical H\"ormander calculus. For example, let ${\mathcal L}_{sub}$ be a sub-Laplacian on $G$, i.e. a sum of squares of left-invariant vector fields satisfying H\"ormander's commutator condition of order $r$. If we consider the Cauchy problem for the corresponding wave equation
\begin{equation}\label{EQ:c2}
\partial_t^2 u-a(t) {\mathcal L}_{sub} u=0,
\end{equation}
even with smooth function $a>0$, it is weakly hyperbolic and its local analysis, while involving the microlocal structure of the sub-Laplacian, is rather complicated. However, the problem \eqref{EQ:c2} can be effectively analysed using the global techniques of pseudo-differential operators on groups. Thus, such results have been obtained in 
\cite{Garetto-Ruzhansky:JDE} with sharp regularity estimates (depending on the H\"ormander commutation order $r$) for the solutions of the Cauchy problem for \eqref{EQ:c2} allowing $a\geq 0$ to be also of H\"older regularity. The $x$-dependent pseudo-differential operators would appear if we allow $a$ to also depend on $x$.

There is a variety of other problems in analysis that require the control of lower order terms of the operator which are not provided by the classical theory of pseudo-differential operators using localisations but can be controlled using the global theory of pseudo-differential operators on groups or on homogeneous spaces. For example, using such techniques, estimates for the essential spectrum of operators on compact homogeneous manifolds have been obtained in \cite{dar:goh}, spaces of Gevrey functions and ultradistributions have been described in \cite{dr:gevrey} relating them with the representation theory of the group acting on the space, Besov and other function spaces have been related to the representation theory of groups in \cite{NRT-Pisa}. It should be noted that many of the developed techniques work not only on groups but also on compact homogeneous manifolds $G/K$ via class I representations of the compact Lie group $G$, thus covering the cases of real, complex or quaternionic spheres, projective spaces, and many other settings. 

Furthermore, many techniques can be extended to non-compact situations, notably those of nilpotent Lie groups, see \cite{FR-book}. Nilpotent Lie groups, in turn, have a wide range of applications to various problems involving differential operators and equations on general manifolds due to the celebrated lifting techniques of Rothschild and Stein \cite{Rothschild-Stein:AM-1976}. We refer to \cite{FR-book} for further explanations and examples of this (nilpotent) setting.



\section{Preliminaries}  
\label{SEC:Prelim}

In this section we recall some basic facts about the theory of pseudo-differential operators on compact Lie groups and we refer to 
\cite{rt:book} and \cite{rt:groups} for a comprehensive account of such topics.  
\medskip

Given a compact Lie group $G$, we equip it with the normalised Haar measure $\mu\equiv dx$ on the Borel $\sigma$-algebra associated to the topology of 
the smooth manifold $G$. The Lie algebra of $G$ will be denoted by $\mathfrak{g}$. We also denote by $\widehat{G}$ the set of equivalence classes of continuous irreducible unitary 
representations of $G$ and by $\Rep(G)$ the set of all such representations. Since $G$ is compact, the set $\widehat{G}$ is discrete.  
For $[\xi]\in \widehat{G}$, by choosing a basis in the representation space of $\xi$, we can view 
$\xi$ as a matrix-valued function $\xi:G\rightarrow \ce^{d_{\xi}\times d_{\xi}}$, where 
$d_{\xi}$ is the dimension of the representation space of $\xi$. 
By the Peter-Weyl theorem the collection
$$
\left\{ \sqrt{d_\xi}\,\xi_{ij}: \; 1\leq i,j\leq d_\xi,\; [\xi]\in\Gh \right\}
$$
is an orthonormal basis of $L^2(G)$.
If $f\in L^1(G)$ we define its global Fourier transform at $\xi$ by 
\begin{equation}\label{EQ:FG}
\mathcal F_G f(\xi)\equiv \widehat{f}(\xi):=\int_{G}f(x)\xi(x)^*dx.
\end{equation}
Thus, if $\xi$ is a matrix representation, 
we have $\widehat{f}(\xi)\in\ce^{d_{\xi}\times d_{\xi}} $. The Fourier inversion formula is a consequence
 of the Peter-Weyl theorem, so that
\beq \label{EQ:FGsum}
f(x)=\sum\limits_{[\xi]\in \widehat{G}}d_{\xi} \Tr(\xi(x)\widehat{f}(\xi)).
\eq
Given a sequence of matrices $a(\xi)\in\mathbb C^{d_\xi\times d_\xi}$, we can define
\begin{equation}\label{EQ:FGi}
(\mathcal F_G^{-1} a)(x):=\sum\limits_{[\xi]\in \widehat{G}}d_{\xi} \Tr(\xi(x) a(\xi)),
\end{equation}
 where the series can be interpreted in a distributional sense or absolutely depending on the growth of 
(the Hilbert-Schmidt norms of) $a(\xi)$. For a further discussion we refer the reader to \cite{rt:book}.

For each $[\xi]\in \widehat{G}$, the matrix elements of $\xi$ are the eigenfunctions for the Laplacian $\mathcal{L}_G$ 
(or the Casimir element of the universal enveloping algebra), with the same eigenvalue which we denote by 
$-\lambda^2_{[\xi]}$, so that
\begin{equation}\label{EQ:Lap-lambda}
-\mathcal{L}_G\xi_{ij}(x)=\lambda^2_{[\xi]}\xi_{ij}(x)\qquad\textrm{ for all } 1\leq i,j\leq d_{\xi}.
\end{equation} 
The weight for measuring the decay or growth of Fourier coefficients in this setting is 
$\jp{\xi}:=(1+\lambda^2_{[\xi]})^{\half}$, the eigenvalues of the elliptic first-order pseudo-differential operator 
$(I-\mathcal{L}_G)^{\half}$.
The Parseval identity takes the form 
\begin{equation}\label{EQ:Parseval}
\|f\|_{L^2(G)}= \left(\sum\limits_{[\xi]\in \widehat{G}}d_{\xi}\|\widehat{f}(\xi)\|^2_{\HS}\right)^{1/2},\quad
\textrm{ where }
\|\widehat{f}(\xi)\|^2_{\HS}=\Tr(\widehat{f}(\xi)\widehat{f}(\xi)^*),
\end{equation}
which gives the norm on 
$\ell^2(\widehat{G})$. 

For a linear continuous operator $A$ from $C^{\infty}(G)$ to $\mathcal{D}'(G) $ 
we define  its {\em matrix-valued symbol} $\sigma(x,\xi)\in\cdxi$ by 
\begin{equation}\label{EQ:A-symbol}
\sigma(x,\xi):=\xi(x)^*(A\xi)(x)\in\cdxi,
\end{equation}
where $A\xi(x)\in \cdxi$ is understood as $$(A\xi(x))_{ij}=(A\xi_{ij})(x),$$ i.e. by 
applying $A$ to each component of the matrix $\xi(x)$.
Then one has (\cite{rt:book}, \cite{rt:groups}) the global quantization
\begin{equation}\label{EQ:A-quant}
Af(x)=\sum\limits_{[\xi]\in \widehat{G}}d_{\xi}\Tr(\xi(x)\sigma(x,\xi)\widehat{f}(\xi))\equiv\sigma(x,D)f(x),
\end{equation}
in the sense of distributions, and the sum is independent of the choice of a representation $\xi$ from each 
equivalence class 
$[\xi]\in \widehat{G}$. If $A$ is a linear continuous operator from $C^{\infty}(G)$ to $C^{\infty}(G)$,
the series \eqref{EQ:A-quant} is absolutely convergent and can be interpreted in the pointwise 
sense. The symbol $\sigma$ can be interpreted as a matrix-valued
function on $G\times\widehat{G}$.
We refer to \cite{rt:book, rt:groups} for the consistent development of this quantization
and the corresponding symbolic calculus. If the operator $A$ is left-invariant then its symbol
$\sigma$ does not depend on $x$. We often call such operators simply invariant.

The following inequality will be useful (e.g. \cite[Theorem 12.6.1]{rt:book}): For $A,B\in \ce^{n\times n}$, we have
\beq\label{ophsi} \|AB\|_{\HS}\leq \|A\|_{op}\|B\|_{\HS},\eq
where $\|\cdot\|_{op}$ denotes the operator norm of the matrix $A$.

Our criteria will be formulated in terms of norms of the matrix-valued symbols. In order to justify their appearance, we recall that if $A\in\Psi_{\rho, \delta}^m(G)$ on a
compact Lie group $G$ is a pseudo-differential operators in H\"ormander's class
$\Psi_{\rho, \delta}^m(G)$, i.e. if all of its localisations to $\Rn$ are pseudo-differential operators with
symbols in the class $S^m_{\rho,\delta}(\Rn)$, then the matrix-symbol of $A$ satisfies
$$\|\sigma(x,\xi)\|_{op}\leq C \jp{\xi}^m \qquad \textrm {for all } x\in G,\; [\xi]\in\Gh.$$
Here $\|\cdot\|_{op}$ denotes the operator norm of the matrix
multiplication by the matrix $\sigma(x,\xi).$
For this fact, see e.g. 
\cite[Lemma 10.9.1]{rt:book} or \cite{rt:groups} in the $(1,0)$ case. For the complete
characterisation of H\"ormander classes $\Psi_{1,0}^m(G)$ in terms of matrix-valued symbols
see also \cite{Ruzhansky-Turunen-Wirth:arxiv}. In particular, this motivates the
usage of the operator norms of the matrix-valued symbols. 

We say that $Q_{\xi}$ is a {\em difference operator} of order $k$ if it is given by
\[Q_{\xi}\widehat{f}(\xi)=\widehat{q_{Q}f}(\xi)\]
for a function $q=q_{Q}\in C^{\infty}(G)$ vanishing of order $k$ at the identity $e\in G$, i.e., $$(P_x q_Q)(e)=0$$ for all left-invariant differential operators $P_x\in \mbox{Diff }^{k-1}(G)$ of order $k-1$. We denote the set of all difference operators of order $k$ by $\mbox{diff }^k(\widehat{G})$. For a given function $q\in C^{\infty}(G)$ it will convenient to denote the associated difference operator, acting on Fourier coefficients, by
\begin{equation}\label{EQ:dq}
\Delta_q\widehat{f}(\xi):=\widehat{qf}(\xi).
\end{equation}

\begin{defn} A collection of $k$ first order difference operators $$\Delta_1,\dots, \Delta_k\in\mbox{diff }^1(\widehat{G})$$ is called {\em admissible}, if the corresponding functions $q_1,\dots, q_k\in C^{\infty}(G)$ satisfy $$\nabla{q_j}(e)\neq 0,\, j=1,\dots, k,$$ 
and 
$$\rank(\nabla{q_1}(e),\dots,\nabla{q_k}(e))=\dim\,G.$$
In particular, the group unit element $e$ is an isolated common zero of the family $\{q_j\}_{j=1}^k$. An admissible collection is called 
{\em strongly admissible} if
\[\bigcap\limits_{j=1}^k\{x\in G:q_j(x)=0\}=\{e\}.\]
\end{defn}
For a given admissible selection of difference operators on a compact Lie group $G$ we use the multi-index notation
\[\Delta_{\xi}^{\alpha}:=\Delta_{1}^{\alpha_1}\cdots\Delta_{k}^{\alpha_k}\mbox{ and }q^{\alpha}(x):=q_{1}(x)^{\alpha_1}\cdots q_{k}(x)^{\alpha_k}.\] 

\begin{defn} Let $\{Y_j\}_{j=1}^{\dim\, G}$ be a basis for the Lie algebra of $G$, and let $\partial_j$ be the left-invariant first order differential operators corresponding to $Y_j$. For $\alpha\in \ene_0^n$, we denote $$\partial^{\alpha}=\partial_1^{\alpha_1}\dots\partial_n^{\alpha_n}.$$ We will use the notation $\partial_x^{\alpha}$ for $\partial^{\alpha}$.
\end{defn}

Let $\Delta_1,\dots, \Delta_k\in\mbox{diff }^1(\widehat{G})$ be a strongly admissible collection 
of first order difference operators.

Let $0\leq\delta\leq\rho\leq 1$ ($\delta<1$). We will say that a matrix-valued symbol $\sigma(x, \xi)$ belongs to $\mathscr{S}_{\rho,\delta}^m(G)$ if it is smooth in $x$,  and for all multi-indices $\alpha, \beta$ there exists  a constant $C_{\alpha,\beta}>0$ such that
\[\|\Delta_{\xi}^{\alpha}\partial_x^{\beta}\sigma(x,\xi)\|_{op}\leq C_{\alpha,\beta}\langle\xi\rangle^{m-\rho|\alpha|+\delta|\beta|},\]
 holds uniformly in $x$ and $\xi\in\Rep(G).$ 
 If $\delta<\rho$, this class is independent of  a strongly admissible collection 
 $\Delta_1,\dots, \Delta_k\in\mbox{diff }^1(\widehat{G})$ of difference operators.
Given a non-negative integer $l$ we associate a seminorm $\|\sigma\|_{l;\mathscr{S}_{\rho,\delta}^m}$ defined by 
\[\|\sigma\|_{l;\mathscr{S}_{\rho,\delta}^m}:=\sup\limits_{|\alpha|+|\beta|\leq l, (x,\xi)}\frac{\|\Delta_{\xi}^{\alpha}\partial_x^{\beta}\sigma(x,\xi)\|_{op}}{\langle\xi\rangle^{m-\rho|\alpha|+\delta|\beta|}}.\]

\begin{rem}\label{eq4a} 
If the group $G$ is viewed as a manifold, the localised H\"ormander class of operators is 
denoted by $\Psi^m_{\rho, \delta}(G, loc)$. In 
 \cite{Ruzhansky-Turunen-Wirth:arxiv} for $(\rho,\delta)=(1,0)$, and then in
\cite[Corollary 8.13]{vf:psecl} for more general $\rho$ and $\delta$, it has been shown that the class of operators $\Psi^m_{\rho, \delta}$ generated by the symbol class $\mathscr{S}_{\rho,\delta}^m(G)$, for $0\leq\delta\leq\rho\leq 1$ ($\delta<1$) coincides with the H\"ormander class of operators $\Psi^m_{\rho, \delta}(G, loc)$ for $0\leq\delta<\rho\leq 1$ and $1-\delta\leq\rho$. \end{rem}

There is also a particular family of difference operators associated to representations that we will need and that we now describe following \cite{Ruzhansky-Wirth:Lp-FAA,Ruzhansky-Wirth:Lp-Z}.
Such difference operators play an important role in the Mihlin multiplier theorem proved in the above papers.

For a fixed irreducible representation $\xi_0$ let us  
define the (matrix-valued) difference operator
$${}_{\xi_0}\mathbb D=
({}_{\xi_0}\mathbb D_{ij})_{i,j=1,\ldots,d_{\xi_0}}$$ 
corresponding to the matrix elements of the matrix-valued function
$\xi_0(x)-\mathrm I$. In other words, if we set $$q_{ij}(x):=\xi_0(x)_{ij}-\delta_{ij}$$ with $\delta_{ij}$ the Kronecker delta,
and use the definition in \eqref{EQ:dq},
then $${}_{\xi_0}\mathbb D_{ij}=\Delta_{q_{ij}}.$$
If the representation is fixed, 
we may omit the index $\xi_0$ for brevity. For a sequence of 
difference operators of this type,
$$\D_1={}_{\xi_1}\D_{i_1 j_1},
\D_2={}_{\xi_2}\D_{i_2 j_2}, \ldots,
\D_k={}_{\xi_k}\D_{i_k j_k},$$ 
with $[\xi_m]\in \widehat G$, $1\leq i_m,j_m\leq d_{\xi_m}$,
$1\leq m\leq k$,
we define
$$\D^\alpha=\D_1^{\alpha_1}\cdots \D_k^{\alpha_k}.$$ 
Let us now fix a particular collection $\Delta_0$ of representations:
Let $\widetilde{\Delta_{0}}$ be the collection of the irreducible components of the adjoint
representation, so that
$${\rm Ad}=(\dim Z(G)) 1\oplus\bigoplus_{\xi\in \widetilde{\Delta_{0}}}\xi,$$
 where $\xi$
are irreducible representations and $1$ is the trivial one-dimensional representation.
In the case when the centre $Z(G)$ of the group is nontrivial, we extend the collection
$\widetilde{\Delta_{0}}$ to some collection $\Delta_{0}$ by adding to $\widetilde{\Delta_{0}}$
a family of irreducible representations such that their direct sum is nontrivial on
$Z(G)$, and such that the function 
\begin{equation*}\label{eq:rhoDef}
\rho^2(x) = \sum\nolimits_{[\xi]\in\Delta_0} \left(
d_\xi-{\rm trace} \xi(x)\right) \ge 0
\end{equation*} 
(which vanishes only in $x=e$) would define the square of some distance function on $G$ 
near the identity element. 
Such an extension is always possible, and we denote by $\Delta_{0}$ any such
extension; in the case of the trivial centre we do not have to take
an extension and we set $\Delta_{0}=\widetilde{\Delta_{0}}$.
We denote further by $\rhodiff$ the second order difference operator associated to $\rho^2(x)$,
$$\rhodiff={\mathscr F} \rho^2(x) {\mathscr F}^{-1}.$$
In the sequel, when we write ${\mathbb D}^\alpha$, we can always
assume that it is composed only of ${}_{\xi_m} {\mathbb D}_{i_m j_m}$ with
$[\xi_m]\in\Delta_0$.

Such difference operators $\D^\alpha$ enjoy a number of additional algebraic properties compared to arbitrary difference operators, for example they satisfy the Leibniz formula, and lead to the distance function $\rho(x)$ that gives rise to the Calderon-Zygmund theory in the spirit of Coifman and Weiss, see \cite{Ruzhansky-Wirth:Lp-Z} for the details.

\section{$L^p$-boundedness}

In this section we establish the main results on the boundedness of operators on $L^p(G)$. We observe that from Theorem A by Fefferman and the equivalence of classes described in Remark \ref{eq4a}, one can extend the Fefferman bounds to compact Lie groups for symbols in $\mathscr{S}_{\rho,\delta}^m(G)$ as an immediate consequence, but assuming that $\delta<\rho$ and $1-\delta\leq\rho$. In particular, this type of argument leads to the restriction $\rho>\half$. So the case of interest to us will be the one allowing $\rho\leq \half$ and $\rho=\delta$. Moreover, we will also obtain 
some $L^p$ bounds for symbols with finite regularity, which can not be deduced from the aforementioned equivalence of classes for smooth symbols. In the latter case we even allow $\delta>\rho$. \\

\subsection{Symbols of finite regularity}

In order to deduce some consequences the following lemma proved in \cite{dr:gevrey} will be of importance to us.
\begin{lem}\label{lem2} 
Let $G$ be a compact Lie group. Then we have 
$$\sum_{[\xi]\in \widehat{G}}d_{\xi}^2\jp{\xi}^{-s}<\infty$$ 
if and only if $s>\dim G$. 
\end{lem}

In the next lemma we deduce a sufficient condition for the $L^{\infty}(G)$-boundedness.
\begin{lem}\label{l1a} Let $G$ be a compact Lie group. Let 
 $\sigma$ be the symbol of a linear continuous operator $A:C^{\infty}(G)\rightarrow\mathcal{D}'(G)$ such that
\beq \esssup_{x\in G}\|\mathcal F_G^{-1}\sigma(x,\cdot)\|_{L^1(G)}<\infty .\label{ji1a}\eq
 Then $A$ extends to a bounded operator from $L^{\infty}(G)$ to $L^{\infty}(G)$,
 and
 \[\|Af\|_{L^{\infty}}\leq \esssup_{x\in G}\|\mathcal F_G^{-1}\sigma(x,\cdot)\|_{L^1(G)}\|f\|_{L^{\infty}},\,\, \forall f\in L^{\infty}(G).\]
\end{lem}
\begin{proof} We first observe that
\begin{align*} Af(x)=&\sum\limits_{[\xi]\in \widehat{G}}d_{\xi}\Tr(\xi(x)\sigma(x,\xi)\widehat{f}(\xi)) \\
=&\int_G\sum\limits_{[\xi]\in \widehat{G}}d_{\xi}\Tr(\xi(x)\sigma(x,\xi)\xi(y)^*)f(y)dy \\
=&\int_G\sum\limits_{[\xi]\in \widehat{G}}d_{\xi}\Tr(\xi(y)^*\xi(x)\sigma(x,\xi))f(y)dy \\
=&\int_G\sum\limits_{[\xi]\in \widehat{G}}d_{\xi}\Tr(\xi(y^{-1}x)\sigma(x,\xi))f(y)dy \\
=&\int_G\mathcal F_G^{-1}\sigma(x,\cdot)(y^{-1}x)f(y)dy \\
=&(f\ast\mathcal F_G^{-1}\sigma(x,\cdot))(x).
\end{align*}
Hence 
\[|Af(x)|\leq \|\mathcal F_G^{-1}\sigma(x,\cdot)\|_{L^1(G)}\|f\|_{L^{\infty}(G)}\leq C\|f\|_{L^{\infty}(G)},\]
for almost every $x\in G$.\\

Therefore
\[\|Af\|_{L^{\infty}(G)}\leq \esssup_{x\in G}\|\mathcal F_G^{-1}\sigma(x,\cdot)\|_{L^1(G)}\|f\|_{L^{\infty}(G)},\]
completing the proof.
\end{proof}

The next statement gives a sufficient condition for an operator to be Hilbert-Schmidt on $L^2(G)$ and bounded on $L^p(G)$ for $2\leq p\leq\infty.$ It also shows an example of when the condition \eqref{ji1a} is satisfied, see also Corollary \ref{the1a}. It will also imply further results on the $L^p$-boundedness. This is a particular feature of the compact situation. 

\begin{prop}\label{cjma1} 
Let $G$ be a compact Lie group. Let $\sigma$ be the symbol of a linear continuous operator $A:C^{\infty}(G)\rightarrow\mathcal{D}'(G)$ such that
\beq\operatorname*{ess\,sup}_{x\in G}\sum\limits_{[\xi]\in \widehat{G}}d_{\xi}\|\sigma(x,\xi)\|_{\HS}^2<\infty .\label{hses}\eq
Then \eqref{ji1a} holds and $A$ extends to a Hilbert-Schmidt bounded operator from $L^{2}(G)$ to $L^{2}(G)$, and to a bounded
 operator from $L^{p}(G)$ to $L^{p}(G)$ for all $2\leq p\leq\infty$.\\
 
 The Hilbert-Schmidt norm of $A$ is given by
\beq\|A\|_{\HS}^2=\int_G\sum\limits_{[\xi]\in \widehat{G}}d_{\xi}\|\sigma(x,\xi)\|_{\HS}^2dx.\label{hs1e}\eq
\end{prop}

\begin{proof} We observe that the Cauchy-Schwarz inequality and the Parseval identity \eqref{EQ:Parseval} imply
\begin{align} \int_G|\mathcal{F}_G^{-1}\sigma(x,\cdot)(y)|dy \leq& \left(\int_G|\mathcal{F}_G^{-1}\sigma(x,\cdot)(y)|^2dy\right)^\half \nonumber\\
=&\left(\sum\limits_{[\xi]\in \widehat{G}}d_{\xi}\|\sigma(x,\xi)\|_{\HS}^2\right)^\half \nonumber\\
\leq& \left(\operatorname*{ess\,sup}_{x\in G}\sum\limits_{[\xi]\in \widehat{G}}d_{\xi}\|\sigma(x,\xi)\|_{\HS}^2\right)^\half <\infty.\label{1eqpr}
\end{align}
Hence by Lemma \ref{l1a} the operator $A$ is bounded from $L^{\infty}(G)$ to $L^{\infty}(G)$.\\


Now, from the proof of Lemma \ref{l1a} we see that the kernel $K_A$ of $A$ is given by 
\[K_A(x,y)=\mathcal{F}_G^{-1}\sigma(x,\cdot)(y^{-1}x).\]
We observe that, by integrating \eqref{1eqpr} over $G$ we get
\begin{align*}\int_G\int_G|K(x,y)|^2dydx=&\int_G\int_G|\mathcal{F}_G^{-1}\sigma(x,\cdot)(y^{-1}x)|^2dydx\\
=&\int_G\int_G|\mathcal{F}_G^{-1}\sigma(x,\cdot)(z)|^2dzdx\\
=&\int_G\sum\limits_{[\xi]\in \widehat{G}}d_{\xi}\|\sigma(x,\xi)\|_{\HS}^2dx<\infty.
\end{align*}
Hence $A$ is a Hilbert-Schmidt operator on $L^2(G)$, \eqref{hs1e} holds and in particular $A$ is bounded on $L^2(G)$. By interpolating between $p=2$ and $p=+\infty$, we conclude the proof.
\end{proof}
\begin{rem} (i) The $\operatorname*{ess\,sup}_{x\in G}$ in the condition \eqref{hses} can be removed if the symbol
$\sigma(x,\xi)$ is continuous on $G$ for all $[\xi]\in \widehat{G}$.\\
(ii) It is not hard to see that an analogous condition to \eqref{hses} does not hold in the case of non-compact groups. 
 Indeed, for $G=\arn$ consider a symbol $\sigma$ of the form $$\sigma(x,\xi)=\beta(x)\alpha(\xi)$$ with $\beta\in L^{\infty}(\arn)\setminus L^{2}(\arn)$, 
$\alpha\in L^{2}(\arn)\setminus \{0\}$. Then, $$\operatorname*{ess\,sup}_{x\in\arn}\int_{\arn}|\sigma(x,\xi)|^2d\xi<\infty,$$ but $\sigma\notin L^2(\arn\times\arn)$
 and $\sigma$ does not beget a Hilbert-Schmidt operator on $L^{2}(\arn)$. 
\end{rem}
As a consequence we obtain the next corollary with a condition in terms of the size of the symbol measured with the operator norm.
\begin{cor}\label{the1a} Let $G$ be a compact Lie group and let $m$ be a real number such that $m>\frac{\dim\, G}{2}$. Let 
 $\sigma$ be the symbol of a linear continuous operator $A:C^{\infty}(G)\rightarrow\mathcal{D}'(G)$ such that
\beq\|\sigma_{A}(x,\xi)\|_{op}\leq C\langle\xi\rangle^{-m},\,\, (x,\xi)\in G\times\widehat{G}.\label{osym}\eq
Then $A$ extends to a Hilbert-Schmidt bounded operator from $L^{2}(G)$ to $L^{2}(G)$, and to a bounded
 operator from $L^{p}(G)$ to $L^{p}(G)$ for all $2\leq p\leq\infty$.
\end{cor}
\begin{proof} By applying \eqref{ophsi} to the decomposition $\sigma(x,\xi)=\sigma(x,\xi)I_{d_{\xi}}$,
 where $I_{d_{\xi}}$ is the identity matrix in $\ce^{d_{\xi}\times d_{\xi}}$, Lemma \ref{lem2} and the assumption on the symbol, we obtain
 for every $x\in G$:
\begin{align*} \sum\limits_{[\xi]\in \widehat{G}}d_{\xi}\|\sigma(x,\xi)\|_{\HS}^2\leq &\sum\limits_{[\xi]\in \widehat{G}}d_{\xi}^2\|\sigma(x,\xi)\|_{op}^2\\
\leq&C\sum\limits_{[\xi]\in \widehat{G}}d_{\xi}^2\langle\xi\rangle^{-2m}<\infty .
\end{align*}
Then, $\sup\limits_{x\in G}\sum\limits_{[\xi]\in \widehat{G}}d_{\xi}\|\sigma(x,\xi)\|_{\HS}^2<\infty,$ 
and an application of Proposition \ref{cjma1} concludes the proof.\end{proof}
\begin{rem}\label{relp1} (i) The condition \eqref{osym} on the order of the symbol, $m>\frac{\dim M}{2}$, is well known (cf. \cite{shubin:r}) in the context of compact manifolds $M$ as a sharp order to ensure that  a pseudo-differential operator is Hilbert-Schmidt. Here, in contrast, we do not assume any smoothness on the symbol nor do we require it to satisfy inequalities for the derivatives.\\

\noindent (ii) We observe that in the case of the torus $\Tn$, if the symbol $\sigma$ only depends on $\xi$ and for $k>\frac{n}{2}$, $k\in\ene$,   
satisfies the inequalities
\beq|\Delta_{\xi}^{\alpha}\sigma(\xi)|\leq C\langle\xi\rangle^{-|\alpha|}, \mbox{ for all }\xi\in\Zn,\label{inesyb1}\eq
and all multi-indices $\alpha$ such that $|\alpha|\leq k$, then the operator $A$ is bounded on $L^p(\Tn)$ for all $1<p<\infty$. 
 Here $$\Delta_{\xi}^{\alpha}=\Delta_{\xi_1}^{\alpha_1}\cdots \Delta_{\xi_n}^{\alpha_n} $$ are the usual partial difference
 operators on the lattice $\Zn$. This classical result has been extended to non-invariant operators on the torus, replacing \eqref{inesyb1} by
 \beq|\partial_x^{\beta}\Delta_{\xi}^{\alpha}\sigma(x,\xi)|\leq C\langle\xi\rangle^{-|\alpha|}, \mbox{ for all }\xi\in\Zn,\label{inesyb2},\eq
and all multi-indices $\alpha, \beta$ such that $|\alpha|\leq k, |\beta|\leq k .$ 


(iii) Moreover, recently in \cite[Theorem 2.1]{Ruzhansky-Wirth:Lp-Z}  a version of the condition \eqref{inesyb1} has been obtained for compact Lie groups. Let $\kappa$ be the smallest even integer larger than $\frac{\dim G}{2}$. Let $A:C^{\infty}(G)\rightarrow\mathcal{D}'(G)$ be a left-invariant linear continuous operator. Among other things it was shown in \cite{Ruzhansky-Wirth:Lp-Z} that if its matrix symbol $\sigma$ satisfies
\beq\|\D_{\xi}^{\alpha}\sigma(\xi)\|_{op}\leq C_{\alpha}\langle\xi\rangle^{-|\alpha|}\eq
for all multi-indices $\alpha$ with $|\alpha|\leq \kappa$ and for all $[\xi]\in \widehat{G}$, then the operator $A$ is of weak  type $(1,1)$ and $L^p$-bounded for all $1<p<\infty.$ 

(iv) Further, a condition of type $(\rho,0)$ has been also obtained in \cite[Corollary 5.1]{Ruzhansky-Wirth:Lp-Z}. Let $\rho\in [0,1]$ and let $\kappa$ be as above. If $A:C^{\infty}(G)\rightarrow\mathcal{D}'(G)$ is left-invariant and its matrix symbol $\sigma$ satisfies
\beq\|\D_{\xi}^{\alpha}\sigma(\xi)\|_{op}\leq C_{\alpha}\langle\xi\rangle^{-\rho|\alpha|}\eq
for all multi-indices $\alpha$ with $|\alpha|\leq \kappa$ and for all $[\xi]\in \widehat{G}$, then the operator $A$ is 
 bounded from the Sobolev space $W^{p,r}(G)$ to $L^p(G)$ for $1<p<\infty$ and
\[r\geq\kappa (1-\rho)\left|\frac{1}{p}-\half\right|.
\]
Here the Sobolev space $W^{p,r}(G)$ consists of all the distributions $f$ such that $(I-\mathcal{L}_G)^{\frac r2}f\in L^p(G)$.\\
(v) Lemma \ref{l1a}, Proposition \ref{cjma1} and Corollary \ref{the1a} admit suitable extensions to general compact topological groups and nilpotent groups, see e.g. \cite{FR:graded}.
\end{rem}

As a consequence of \cite[Corollary 5.1]{Ruzhansky-Wirth:Lp-Z} (or Remark \ref{relp1} (iv)), we obtain: 
\begin{thm}\label{lepr} Let $\rho\in [0,1]$ and let $\kappa$ be the smallest even integer larger than $\frac{\dim G}{2}$. If $A:C^{\infty}(G)\rightarrow\mathcal{D}'(G)$ is left-invariant and its matrix symbol $\sigma$ satisfies
\beq\label{leq1s}\|\D_{\xi}^{\alpha}\sigma(\xi)\|_{op}\leq C_{\alpha}\langle\xi\rangle^{-r-\rho|\alpha|}\eq
with $$r=\kappa(1-\rho)\left|\frac{1}{p}-\half\right|,$$ for all multi-indices $\alpha$ with $|\alpha|\leq \kappa$ and for all $[\xi]\in \widehat{G}$, $1<p<\infty$, then the operator $A$ is bounded from $L^{p}(G)$ to $L^p(G)$.
\end{thm}
\begin{proof} We observe that the symbol
\[\Gamma(\xi)=\sigma(\xi)\langle\xi\rangle^{r}\]
satisfies
\[\|\D_{\xi}^{\alpha}\Gamma(\xi)\|_{op}\leq C_{\alpha}\langle\xi\rangle^{-\rho|\alpha|},\]
  for all multi-indices $\alpha$ with $|\alpha|\leq \kappa$ and for all $[\xi]\in \widehat{G}$. Then, the left-invariant operator ${\rm Op}(\Gamma)$
 corresponding to $\Gamma$ is bounded from $W^{p,r}(G)$ into $L^p(G)$ with $r=\kappa(1-\rho)\left|\frac{1}{p}-\half\right|$. But 
$${\rm Op}(\Gamma)=A(I-\lap)^{\frac{r}{2}}$$ and $(I-\lap)^{\frac{r}{2}}$ is an isomorphism between $W^{p,r}(G)$ and $L^p(G)$. 
Therefore $A$ is bounded from $L^{p}(G)$ into $L^p(G)$.
\end{proof}
We note that if the condition \eqref{leq1s}  holds for $\rho=1$, then $r=0$, and $A$ is bounded on $L^p(G)$ for every $1<p<\infty$. Hence, Theorem \ref{lepr} absorbs the condition \eqref{leq1s}. We now derive a $(\rho,\delta)$-type condition following the main idea in the proof of Theorem 5.2 in \cite{Ruzhansky-Wirth:Lp-Z} and using Theorem \ref{lepr}.

\begin{thm} \label{thpr} 
Let $0\leq\delta,\rho\leq 1$ and $\dim G=n$. Denote by $\kappa$ the smallest even integer larger than $\frac{n}{2}$.  Let $1<p<\infty$ and let $\ell>\frac{n}{p}$ be an integer. If $A:C^{\infty}(G)\rightarrow\mathcal{D}'(G)$ is a linear continuous operator such that its matrix symbol $\sigma$ satisfies
\beq\label{aieq1}\|\partial_x^{\beta}\D_{\xi}^{\alpha}\sigma(x,\xi)\|_{op}\leq C_{\alpha, \beta}\langle\xi\rangle^{-m_0-\rho|\alpha|+\delta|\beta|}\eq
for all $x$, with  
$$m_0\geq\kappa(1-\rho)\left|\frac{1}{p}-\half\right|+\delta([\frac{n}{p}]+1),$$ 
for all multi-indices $\alpha, \beta$ with $|\alpha|\leq \kappa$, $|\beta|\leq\ell$ and for all $[\xi]\in \widehat{G}$, then the operator $A$ is bounded from $L^{p}(G)$ to $L^p(G)$. 
\end{thm}
\begin{proof} We first write $$Af(x)=(f\ast r_{A}(x))(x),$$ where
\[r_{A}(x)(y)=R_A(x,y)\]
denotes the right-convolution kernel of $A$. Let
\[A_yf(x):=(f\ast r_{A}(y))(x),\]
so that $A_xf(x)=Af(x).$

\medskip

Now we see that 
\[\|Af\|_{L^p(G)}^p=\int_G|A_xf(x)|^pdx\leq \int_G\sup\limits_{y\in G}|A_yf(x)|^pdx.\]

By applying the Sobolev embedding theorem, we obtain
\[\sup\limits_{y\in G}|A_yf(x)|^p\leq C\sum\limits_{|\gamma|\leq \ell}\int_G|\partial_y^{\gamma}A_yf(x)|^pdy.\]
Now, by Fubini Theorem we have
\begin{align*} \|Af\|_{L^p(G)}^p\leq &C \sum\limits_{|\gamma|\leq \ell}\int_G\int_G|\partial_y^{\gamma}A_yf(x)|^pdxdy \\
\leq &C \sum\limits_{|\gamma|\leq \ell}\sup\limits_{y\in G}\int_G|\partial_y^{\gamma}A_yf(x)|^pdx \\
= &C \sum\limits_{|\gamma|\leq \ell}\sup\limits_{y\in G}\|\partial_y^{\gamma}A_yf\|_{L^p(G)}^p \\
\leq &C \sum\limits_{|\gamma|\leq \ell}\sup\limits_{y\in G}\|f\mapsto f\ast\partial_y^{\gamma}r_A(y)\|_{\mathcal{L}(L^p(G))}^p\|f\|_{L^p(G)}^p. 
\end{align*}
Thus, the operator $A$ will be bounded on $L^p(G)$ provided that the left-invariant operators $$f\mapsto f\ast\partial_y^{\gamma}r_A(y)$$ are uniformly  bounded on $L^p(G)$ with respect to $y\in G, |\gamma|\leq\ell$. We shall now estimate for each  $y\in G$, the conditions under which such operators are bounded on $L^p(G)$ according to Theorem \ref{lepr}. The symbol of the left-invariant operator  $D_{y, \gamma}:f\mapsto f\ast\partial_y^{\alpha}r_A(y)$ is given by $$\sigma_{D_{y, \gamma}}(\xi)=\partial_y^{\gamma}\sigma(y,\xi).$$
From \eqref{aieq1} we have
\beq\label{iha1}\|\D_{\xi}^{\alpha}\sigma_{D_{y, \gamma}}(\xi)\|_{op}\leq C_{\alpha, \beta}\langle\xi\rangle^{-m_0+\delta|\gamma|-\rho|\alpha|},\eq
for all multi-indices $\alpha, \beta$ with $|\alpha|\leq \kappa$, $|\gamma|\leq\ell$ and for all $[\xi]\in \widehat{G}$.\\

Hence $D_{y, \gamma}$ is bounded on $L^p(G)$ provided that $$-m_0+\delta|\gamma|\leq -r=-\kappa(1-\rho)\left|\frac{1}{p}-\half\right|$$ and $|\gamma|\leq \ell$. But this follows from the condition $$m_0\geq\kappa(1-\rho)\left|\frac{1}{p}-\half\right|+\delta([\frac{n}{p}]+1),$$ since $\ell\geq [\frac{n}{p}]+1$ and \eqref{iha1} holds for $|\gamma|\leq [\frac{n}{p}]+1$. 
Then, the operators $D_{y, \gamma}$ are uniformly bounded on $L^p(G)$ with respect to $y\in G, |\gamma|\leq\ell$, which concludes the proof.
\end{proof}

\begin{rem} (i) If in Theorem \ref{thpr} the operator $A$ is left-invariant, then the symbol $\sigma$ depends only on $\xi$ and 
 the conditions of the theorem recover those of Theorem \ref{lepr}.\\
(ii) In Theorem \ref{thpr} the usual condition $\delta\leq\rho$ is not imposed, and in particular $\delta>\rho$ is allowed.\\
(iii) For relatively large values of $p$, for instance if $p>n$, the Theorem \ref{thpr} only requires smoothness of the first order with respect to $x$ for the symbol $\sigma$, i.e. reduces to conditions on first order derivatives only. \\
(iv) By (ii) and (iii), in the extreme situation $\rho=0, \delta=1$, and for $p>n$, the condition \eqref{aieq1} with $m_0=\kappa(1-\rho)\left|\frac{1}{p}-\half\right|+\delta([\frac{n}{p}]+1)$ takes the form 
\beq\label{aieq1c}\|\partial_x^{\beta}\D_{\xi}^{\alpha}\sigma(x,\xi)\|_{op}\leq C_{\alpha, \beta}\langle\xi\rangle^{-\kappa\left|\frac{1}{p}-\half\right| +(|\beta|-1)}\eq
for all multi-indices $\alpha, \beta$ with $|\alpha|\leq \kappa$, $|\beta|\leq 1$ and for all $[\xi]\in \widehat{G}$. 

We note that in such situation, the required regularity $(=1)$ is independent of the dimension $\dim G=n$, which is in contrast to the situation in the Euclidean setting. In particular, that is the case of the finite regularity improved version of Fefferman's bounds obtained by Li and Wang in \cite{wa-li:lp}. 

For values near $p=2$ the situation is opposite and the condition in Theorem \ref{thpr} does not improve C. Fefferman's type conditions that can be obtained for $C^{\infty}$-smooth symbols.  We will obtain sharper conditions on the symbol but requiring $C^{\infty}$-smoothness. In particular the usual restriction $\rho\geq\half$ for $(\rho, \delta)$ classes on manifolds will not be imposed here as an advantageous consequence of the global symbolic calculus  on compact Lie groups at our disposal. 
\end{rem}

\begin{rem}
Recently, the Mihlin multiplier theorem obtained by the second author and Wirth in \cite{Ruzhansky-Wirth:Lp-FAA,Ruzhansky-Wirth:Lp-Z} has been reobtained by Fischer in
\cite{Fischer-Horm} using different collections of difference operators. If the integer part of $n/2$ is odd, the orders of required difference operators coincide, while if it is even, the order is improved by one. However, at the same time, the collections of difference operators one has to work with are different: the difference operators in \cite{Fischer-Horm} come from fundamental representations of the group, while our collection $\D^\alpha$ comes from the finite decomposition of the adjoint representation into irreducible components. However, since they are related, it is probable that the evenness of $\kappa$ can be removed for our collection of difference operators $\D^\alpha$.

In any case, if in \eqref{aieq1} one replaces the collection $\D^\alpha$ of difference operators by the collection of difference operators associated to fundamental representations of the group, a simple modification of the proof yields the statement of Theorem \ref{thpr} 
with $\kappa$ being the smallest integer larger than $\frac{n}{2}$ without requiring its evenness, giving an improvement of the order by one for half of the dimensions. 
\end{rem}

\subsection{$C^{\infty}$-smooth symbols}

We turn now to a different perspective by looking for conditions for $C^{\infty}$-smooth symbols. We will employ the geodesic distance on the group $G$ and it will be denoted  by $d$. For the distances from the unit element $e$ we will write $|y|=d(y,e)$. The corresponding $BMO$ space with respect to this distance will
 be denoted by $BMO(G)$. The following lemma will be useful to obtain  
 $L^{\infty}-BMO(G)$ bounds by applying partitions of unity.
 
Here and in the sequel it will be also useful to introduce the number
$$a:=1-\rho$$
 that we will use everywhere without special notice. Before formulating the following lemma
 we record asymptotic properties that will be of use on several occasions:
 asymptotically as $\lambda\to\infty$  we have
\begin{equation}\label{EQ:dims1}
\sum_{\jp{\xi}\leq\lambda}d_{\xi}^{2} \jp{\xi}^{\alpha n}\asymp \lambda^{(\alpha+1)n}
\; \textrm{ for }\; \alpha>-1,
\end{equation}
and
\begin{equation}\label{EQ:dims2}
\sum_{\jp{\xi}\geq\lambda}d_{\xi}^{2} \jp{\xi}^{\alpha n}\asymp \lambda^{(\alpha+1)n}
\; \textrm{ for }\; \alpha<-1.
\end{equation}
We refer to \cite{ar:pq1} for their proof.
 
\begin{lem}\label{l1wa} Let $G$ be a compact Lie group of dimension $n$ and let $0<a<1$. Let $0\leq\delta<1-a$. Let  $\sigma\in \mathscr{S}_{1-a,\delta}^{-\frac{na}{2}}(G)$ be supported in
\[\{(x,\xi)\in G\times\widehat{G}: R\leq \langle\xi\rangle\leq 3R\} ,\]
for some $R>0$. Then $\sigma(x,D)$ extends to a bounded operator from $L^{\infty}(G)$ to $L^{\infty}(G)$, and for $l\geq \frac{n}{2}$ we have
\[\|\sigma(x,D)f\|_{L^{\infty}(G)}\leq C\|\sigma\|_{l,\mathscr{S}_{1-a,\delta}^{-\frac{na}{2}}}\|f\|_{L^{\infty}(G)},\]
with $C$ independent of $\sigma$, $f$ and $R$.
\end{lem}
\begin{proof} Let $\sigma\in \mathscr{S}_{\rho,\delta}^{-m}(G)$ be supported in
 \beq\{(x,\xi)\in G\times\widehat{G}: R\leq \langle\xi\rangle\leq 3R \},\label{iqsu}\eq
for some fixed $R>0$. In order to prove Lemma \eqref{l1wa} we will apply Lemma \ref{l1a}. We split $G$ into the form $$G=\{y\in G: |y|\leq b\}\cup \{y\in G: |y|> b\},$$ where $b=R^{a-1}$.

By applying Cauchy-Schwarz inequality, Parseval identity and the inequality \eqref{ophsi} to the decomposition $$\sigma(x,\xi)=\sigma(x,\xi)I_{d_{\xi}}$$  we obtain 
\begin{align}\int_{|y|\leq b}|\mathcal{F}_G^{-1}\sigma(x,\cdot)(y)|dy &\leq  \left(\mu(\{|y|\leq b\})\right)^{\half} \left(\int_G|\mathcal{F}_G^{-1}\sigma(x,\cdot)(y)|^2dy\right)^\half \nonumber\\
&=\left(\mu(\{|y|\leq R^{a-1}\})\right)^{\half}\left(\sum\limits_{[\xi]\in \widehat{G}}d_{\xi}\|\sigma(x,\xi)\|_{\HS}^2\right)^\half \nonumber\\
&\leq CR^{\frac{n(a-1)}{2}}\left(\sum\limits_{\{R\leq \langle\xi\rangle\leq 3R\} }d_{\xi}\|\sigma(x,\xi)\|_{\HS}^2\right)^\half \nonumber\\
&\leq  CR^{\frac{n(a-1)}{2}}\left(\sum\limits_{\{R\leq \langle\xi\rangle\leq 3R\} }d_{\xi}^2\|\sigma(x,\xi)\|_{op}^2\right)^\half \nonumber\\
&\leq C\|\sigma\|_{0,\mathscr{S}_{1-a,\delta}^{-\frac{na}{2}}}R^{\frac{n(a-1)}{2}}\left(\sum\limits_{\{R\leq \langle\xi\rangle\leq 3R\} }d_{\xi}^2\langle\xi\rangle^{-na}\right)^\half \nonumber\\
&\leq  C\|\sigma\|_{0,\mathscr{S}_{1-a,\delta}^{-\frac{na}{2}}}R^{\frac{n(a-1)}{2}}R^{\frac{n(1-a)}{2}} \label{inrm}\\
&\leq C\|\sigma\|_{0,\mathscr{S}_{1-a,\delta}^{-\frac{na}{2}}} \,<\infty,\label{1eqpra}
\end{align}
with $C$ independent of $R$ and $\sigma$. For the inequality \eqref{inrm} we have applied the estimate \eqref{EQ:dims1}.

We now consider the integral $\int_{|y|\geq b}|\mathcal{F}_G^{-1}\sigma(x,\cdot)(y)|dy$. To analyse it we take the difference operator $\Delta_q$ associated to $q$  that vanishes at $e$ of order $l$ and $e$ is its isolated zero, i.e., there exist constants $C_1, C_2>0$ such that
\[C_1|y|^{l}\leq |q(y)|\leq C_2|y|^l.\] 

We first note that $|q(y)|\leq C|y|^l$, for small $|y|$, e.g. $|q(y)|\leq Cd^l$,
 for $|y|\leq d$ for some suitable $d$. We have, using the boundedness of $r$,
\begin{align} \int_{|y|\geq b}|\mathcal{F}_G^{-1}\sigma(x,\cdot)(y)|dy &=\int_{|y|\geq b}\frac{|q(y)(\mathcal{F}_G^{-1}\sigma(x,\cdot))(y)|}{|q(y)|}dy\nonumber\\
&\leq\left(\int\limits_{\{|y|\geq b\}}|q(y)|^{-2} dy\right)^{\half} \left(\int_G|q(y)(\mathcal{F}_G^{-1}\sigma(x,\cdot))(y)|^2dy\right)^\half \nonumber\\
&\leq C\left(\int\limits_{\{|y|\geq b\}}|y|^{-2l}dy\right)^{\half}\left(\sum\limits_{\{R\leq \langle\xi\rangle\leq 3R\} }d_{\xi}\|\Delta_q\sigma(x,\xi)\|_{\HS}^2\right)^\half \nonumber\\
&\leq  C(b^{n-2l})^{\half}\left(\sum\limits_{\{R\leq \langle\xi\rangle\leq 3R\}}d_{\xi}^2\|\Delta_q\sigma(x,\xi)\|_{op}^2\right)^\half \nonumber\\
&\leq C\|\sigma\|_{l,\mathscr{S}_{1-a,\delta}^{-\frac{na}{2}}}R^{(a-1)(\frac{n}{2}-l)}\left(\sum\limits_{\{R\leq \langle\xi\rangle\} }d_{\xi}^2\langle\xi\rangle^{-na-2l(1-a)}\right)^\half \nonumber\\
&\leq  C\|\sigma\|_{l,\mathscr{S}_{1-a,\delta}^{-\frac{na}{2}}}R^{(a-1)(\frac{n}{2}-l)}R^{(1-a)(\frac{n}{2}-l)} \label{1eqprac}\\
&\leq C\|\sigma\|_{l,\mathscr{S}_{1-a,\delta}^{-\frac{na}{2}}}<\infty ,\nonumber
\end{align}
with $C$ independent of $R$ and $\sigma .$ For the estimation of the integral $\int\limits_{\{|y|\geq b\}}|y|^{-2l}dy$ we note that the essential case
 is $b$ small, and so the bound can be reduced  to a local estimation. 
\end{proof}

We will now establish a  $L^{\infty}(G)-BMO(G)$ estimate which will have as a consequence the main results
 for smooth symbols.  The space $BMO(G)$ correspond to the system of balls $B(x,r)$ determined by the geodesic distance $d$. 
 
\begin{thm}\label{mth1a} 
Let $G$ be a compact Lie group of dimension $n$ and let $0<\rho<1$. Let $0\leq\delta<\rho$ and $\sigma\in \mathscr{S}_{\rho,\delta}^{-\frac{n(1-\rho)}{2}}(G)$. Then $\sigma(x,D)$ extends to a bounded operator from $L^{\infty}(G)$ to $BMO(G)$ and moreover, for $l>\frac n2$ we have 
\[\|\sigma(x,D)f\|_{BMO(G)}\leq C\|\sigma\|_{l,\mathscr{S}_{\rho,\delta}^{-\frac{n(1-\rho)}{2}}}\|f\|_{L^{\infty}(G)},\]
with $C$ independent of $\sigma$ and $f$.
\end{thm}
\begin{proof} 
Here and everywhere we will write $a:=1-\rho$.
Let us fix $f\in L^{\infty}(G)$ and $B=B(x_0,r)\subset G$. We will show that there exist an integer $k$ and a constant $C>0$ independents of $f$ and $B$ such that 
\beq
 \frac{1}{\mu(B(x_0,r))}\int\limits_{B}|\sigma(x,D)f(x)-g_B|dx\leq
C \|\sigma\|_{k,\mathscr{S}_{1-a,\delta}^{-\frac{na}{2}}}\|f\|_{L^\infty} ,\label{eq:pbmo}
\eq
where we have written $g=\sigma(x,D)f$ and $g_B=\frac{1}{\mu(B(x_0,r))}\int\limits_{B}gdx.$\\

We also write $$R_0=\sup\{R: \exists x\in G\, \textrm{ such that }
 B(x,R)\subset G\}.$$ We split $\sigma(x,\xi)$ into two symbols,
$$\sigma=\sigma^0 + \sigma^1,$$ with $\sigma^0 $ supported in 
$\langle\xi\rangle\leq 2R_0r^{-1}$ and $\sigma^1$ supported in 
$\langle\xi\rangle\geq \half R_0r^{-1}$, satisfying the following estimates

\beq\label{eq:dec1}\|\sigma^0\|_{l,\mathscr{S}_{1-a,\delta}^{-\frac{na}{2}}},\,\|\sigma^1\|_{l,\mathscr{S}_{1-a,\delta}^{-\frac{na}{2}}}\leq C_l\|\sigma\|_{l,\mathscr{S}_{1-a,\delta}^{-\frac{na}{2}}}, \mbox{ for every
}l\geq 1.\eq

In order to establish the existence of the above splitting (\ref{eq:dec1}) one can consider a function 
$0\leq\gamma\in C^\infty(\ar)$ 
 wich equals to $1$ if $|t|\leq \half$ and with ${\mbox{ supp}} \,\gamma=\{|t|\leq 1\}$. We set
\[\tilde{\gamma}(\xi)=\gamma(r\langle\xi\rangle)\]
 and 
\[\sigma^0(x,\xi)=\sigma(x,\xi)\tilde{\gamma}(\xi). \]
Moreover the seminorms of $\sigma^0$ can be controlled by those of $\sigma$ by
\beq\|\sigma^0\|_{l,\mathscr{S}_{1-a,\delta}^{-\frac{na}{2}}}\leq C_l\|\sigma\|_{l,\mathscr{S}_{1-a,\delta}^{-\frac{na}{2}}}, \mbox{ for all
}l\geq 1.\label{ikbf}\eq
Now, by taking $\sigma^1=\sigma-\sigma^0$, the estimate \eqref{ikbf} is still valid for $\sigma^1$.

Now we note that for a left invariant vector field $X$ on $G$ we have
\begin{align*}X(\xi(x)\sigma(x,\xi))&=X(\xi(x))\sigma(x,\xi)+\xi(x)X\sigma(x,\xi)\\
&=\xi(x)\sigma_X(\xi)\sigma(x,\xi)+\xi(x)X\sigma(x,\xi),
\end{align*}
where we have used the fact that $$\xi(x)\sigma_X(\xi)=(X\xi)(x)$$ from \eqref{EQ:A-symbol} applied to $A=X$.

Hence
\[ XAf(x)=\sum\limits_{[\xi]\in\Gh}d_{\xi}\Tr \left(\xi(x)(\sigma_X(\xi)\sigma(x,\xi)+X\sigma(x,\xi))\widehat{f}(\xi)\right) .\]

In particular, with $X=X_k=\partial_{x_k}$ being a left-invariant vector field we obtain
\begin{equation}\label{EQ:sigmas}
\partial_{x_k}\sigma^0(x,D)f(x)=\sigma'(x,D)f(x),
\end{equation}
where
\[\sigma'(x,\xi)=\sigma_{X_k}(\xi)\sigma(x,\xi)+X_k\sigma(x,\xi).\]
By using a suitable partition of unity we write
\[\sigma'(x,\xi)=\sum\limits_{j=1}^{\infty}\rho_j(x,\xi),\]
with $\rho_j$ supported in $\langle\xi\rangle\sim 2^{-j}r^{-1}$, and such that
\beq\|\rho_j\|_{l,\mathscr{S}_{1-a,\delta}^{-\frac{na}{2}}}\leq C2^{-j}r^{-1}\|\sigma\|_{l,\mathscr{S}_{1-a,\delta}^{-\frac{na}{2}}}.\label{itykbf}\eq

 In order to construct such partition of unity, we consider $\eta:\ar\rightarrow\ar$ defined by
\[\eta(t)=\left\{\begin{array}{ll}
0&,\mbox{ if }|t|\leq 1 , \\
1&,\mbox{ if }|t|\geq 2.
\end{array}
\right.  \] 
\indent We put $\rho(t)=\eta(t)-\eta(2^{-1}t). $ Then 
\[{\mbox{ supp}} \,\rho=\{1\leq |t|\leq 4\}.\]
One can see that
\[1=\eta(t)+\sum\limits_{j=1}^\infty\rho(2^jt)\, ,\mbox{for all
}t\in\ar .\]
Indeed,
\[\eta(t)+\sum\limits_{j=1}^\ell\rho(2^jt)=\]
\[=\eta(t)+\eta(2t)-\eta(t)+\eta(2^2t)-\eta(2t)+\cdots+\eta(2^{\ell}t)-\eta(2^{\ell-1}) \]
\[=\eta(2^{\ell}t)\rightarrow 1 \mbox{ as }\ell\rightarrow\infty.\]
\indent 
In particular we can write $t=r\langle\xi\rangle$ and then
\[1=\eta(r\langle\xi\rangle)+\sum\limits_{j=1}^\infty\rho(r 2^j\langle\xi\rangle). \]
\indent The support of $\eta$ is  $\{|t|>1\}$ and if 
$r\langle\xi\rangle\leq 1$, then
 \[\eta(r\langle\xi\rangle)\equiv 0\]
and hence
\[1=\sum\limits_{j=1}^\infty\rho(r2^j\langle\xi\rangle)\]
for $\{\langle\xi\rangle\leq r^{-1}\}.$\\
\indent Now, since ${\mbox{ supp }} \,\sigma'={\mbox{ supp }} \,\sigma^0$ in view of \eqref{EQ:sigmas}, we have
\[\sigma'(x,\xi)=\sum\limits_{j=1}^\infty\rho(r 2^j\langle\xi\rangle)\cdot\sigma'(x,\xi).\]
\indent We set 
\[\rho_j(x,\xi)=\rho(r 2^j\langle\xi\rangle)\cdot\sigma'(x,\xi)=\alpha_j(\xi)\sigma'(x,\xi). \]
The estimate \eqref{itykbf} for the seminorms of $\rho_j$ now follows from the relation between $\sigma$ and $\sigma'$ as well as from the fact that $\rho_j(x,\xi)=\alpha_j(\xi)\sigma'(x,\xi)$. 

We now apply  Lemma \ref{l1wa} to the pieces $\rho_j$ obtaining 
\begin{align*} 
 \|\partial_{x_k}\sigma^0 (x,D)f\|_{L^\infty}&\leq
 \sum\limits_{j=0}^\infty\|\rho_j(x,D)f\|_{L^\infty}\\
&\leq Cr^{-1} \sum\limits_{j=0}^\infty
2^{-j}\|\sigma\|_{l,\mathscr{S}_{1-a,\delta}^{-\frac{na}{2}}} \|f\|_{L^\infty}\\
&\leq C r^{-1}\|\sigma\|_{l,\mathscr{S}_{1-a,\delta}^{-\frac{na}{2}}} \|f\|_{L^\infty}.
\end{align*}
An application of the Mean Value Theorem gives us
\[|\sigma^0(x,D)f(x)-g_B|\leq C\|\sigma\|_{l;\mathscr{S}_{1-a,\delta}^{-\frac{na}{2}}}
\|f\|_{L^\infty}.\]
Hence
\beq\frac{1}{|B(x_0,r)|}\int\limits_{B}|\sigma^0(x,D)f(x)-g_B|dx\leq
C \|\sigma\|_{l,\mathscr{S}_{1-a,\delta}^{-\frac{na}{2}}} \|f\|_{L^\infty} ,\eq
which gives  (\ref{eq:pbmo}) for $\sigma^0 $.

We now consider the term $\sigma^1$, we recall that we fixed a ball $B(x_0,r)\subset G$. We now also fix a cut-off function  $\phi$
 over $G$, with  $0\leq\phi\leq 10$, $ \phi\geq 1$ on 
 $B(x_0,r)$ and such that its Fourier transform 
$\widehat{\phi}$ verifies ${\mbox{supp } }(\widehat{\phi})\subset \{\langle\xi\rangle\leq (C^{-1}r)^{\frac{1}{1-a}}\}$. 
 Let us write
\beq\,\mbox{ }\mbox{ }\mbox{ }\,\,\,\phi(x)\cdot\sigma^1(x,D)f(x)=\sigma^1(x,D)(\phi
f)(x)+\left[\phi,\sigma^1(x,D)\right]f(x)= I+II \label{eq:dec}.\eq
\indent For the estimation of {\em I} we begin by factorising it in a suitable way. Let
 $L$ be the following power of the Laplacian on $G$, $L=(1-\lap)^{\frac{na}{2}}$.  By  \cite[Theorem 4.2]{Ruzhansky-Wirth:functional-calculus} 
 we have $L\in \Op\mathscr{S}_{1,0}^{na}(G)$ and it is known that $L^{-1}$ is a positive operator (cf. \cite{rs:lman}, \cite{g:man2}, \cite{g:man}), i.e., $L^{-1}(g)\geq 0$ if $g\geq
0$.  We write
\beq\sigma^1(x,D)(\phi f)=\left(\sigma^1(x,D)\circ
  L\right)\left( L^{-1}(\phi f)\right).\eq

Since $\Op\mathscr{S}_{1,0}^{\frac{na}{2}}(G)\subset \Op\mathscr{S}_{1-a,\delta}^{\frac{na}{2}}(G)$, we note
 that  $\sigma^1(x,D)\circ L$ is a pseudo-differential operator in $\Op\mathscr{S}_{1-a,\delta}^{0}$.
 We also have 
\beq L^{-1}:H^{-\frac{na}{2}}\longrightarrow L^2.\label{eq:normeL}\eq
\indent By the $L^2$ boundedness for operators of order $0$ (cf. \cite[Proposition 8.1]{vf:psecl}),
 applied to the operator  $\sigma^1(x,D)\circ L$, we deduce the existence 
 of a constant $C$ and an integer $l_0$ such that
\beq\|\sigma^1(x,D)(\phi f)\|_{L^2}^2\leq C \|\sigma
_A^1\|_{l_0;\mathscr{S}_{1-a,\delta}^{-\frac{na}{2}}}^2\cdot \|L^{-1}(\phi f)\|_{L^2}^2  .\eq
On the other hand, by \eqref{eq:normeL} we have 

\[\|L^{-1}(\phi f)\|_{L^2}^2\leq C\|\phi
f\|_{H^{-\frac{na}{2}}}^2.\]
We also observe that $\|\widehat{\phi}(\xi)\|_{\HS}^2\leq C'd_{\xi}$. Indeed,
 by the definition of the Fourier transform on compact groups, since $\xi(x)$ is unitary and $0\leq\phi\leq 10$ we have
\[\|\widehat{\phi}(\xi)\|_{\HS}\leq\int_G \|\xi^*(x)\|_{\HS}|\phi(x)|dx\leq 10\sqrt{d_{\xi}},\]
where we have used the identities $$\|\xi^*(x)\|_{\HS}^2=\Tr(\xi^*(x)\xi(x))=\Tr(I_{d_{\xi}})=d_{\xi}.$$

Since  $L$ is a positive operator, we obtain
\begin{eqnarray*}
\|L^{-1}(\phi f)\|_{L^2}^2&\leq \|f\|_{L^\infty}^2
\|L^{-1}(\phi)\|_{L^2}^2\\
&\leq C_1\|f\|_{L^\infty}^2\|\phi\|_{H^{-\frac{na}{2}}}^2\\
&\leq C_1\|f\|_{L^\infty}^2(C^{-1}r)^{n}\\
&\leq C\|f\|_{L^\infty}^2 |B(x_0,r)|.
\end{eqnarray*}
For the estimation of $\|\phi\|_{H^{-\frac{na}{2}}}^2$ we  have used the following inequalities:
 \begin{eqnarray*}
\|\phi\|_{H^{-\frac{na}{2}}}^2&\leq \sum\limits_{\langle\xi\rangle\leq Cr^{\frac{1}{1-a}}}d_{\xi}\langle\xi\rangle ^{-na}\|\widehat{\phi}(\xi)\|_{\HS}^2\\
&\leq C'\sum\limits_{\langle\xi\rangle\leq Cr^{\frac{1}{1-a}}}d_{\xi}^2\langle\xi\rangle ^{-na}\\
&\leq C_2r^{\frac{1}{1-a}(1-a)n}=C_2r^n.
\end{eqnarray*}
For the last inequality we have applied the estimate \eqref{EQ:dims1}.\\

Thus
\beq\|\sigma^1(x,D)(\phi f)\|_{L^2}^2\leq C \|\sigma^1\|_{l_0;\mathscr{S}_{1-a,\delta}^{-\frac{na}{2}}}^2\|f\|_{L^\infty}^2 |B(x_0,r)|.\eq
By the Cauchy-Schwarz inequality we get
\beq\frac{1}{|B(x_0,r)|}\int\limits_{B}|\sigma^1(x,D)(\phi
f)(x)|dx\leq\label{eq:s}\eq
$$\leq\left(\frac{1}{|B|}\int\limits_{B}|\sigma^1(x,D)(\phi  f)(x)|^2 dx\right)^\half$$
$$ \leq C \|\sigma^1\|_{l_0;\mathscr{S}_{1-a,\delta}^{-\frac{na}{2}}}\|f\|_{L^\infty}. $$
\noi This proves the desired estimated for $I$.\\

\indent For the estimation of $II$, we begin by writing $\left[\phi,\sigma^1(x,D)\right]f(x) $ in the form 
\[\theta (x,D)f(x),\] 
where $\theta (x,\xi)$ is a suitable symbol. To calculate $\theta(x,\xi)$ we write $B:=\sigma^1(x,D)$, consider the convolution 
 kernel $k_x$ of $B$ and observe that
\begin{align*}
\left[\phi,\sigma^1(x,D)\right]f(x)&=\phi(x)Bf(x)-B(\phi f)(x)\\
&=\phi(x)(f\ast k_x)(x)-((\phi f)\ast k_x)(x)\\
&=\int\limits_G\phi(x)f(xy^{-1})k_x(y)dy-\int\limits_G\phi(xy^{-1})f(xy^{-1})k_x(y)dy\\
&=\int\limits_Gf(xy^{-1})\left(\phi(x)k_x(y)-\phi(xy^{-1})k_x(y)\right)dy.\\
\end{align*}
Hence, $\theta(x,\xi)$ is given by 
\[\theta(x,\xi)=\int\limits_G\left(\phi(x)k_x(y)-\phi(xy^{-1})k_x(y)\right)\xi (y)^*dy.\]

On the other hand, by using Taylor expansions on compact Lie groups (see e.g. \cite{rt:book} or \cite{Ruzhansky-Turunen-Wirth:arxiv}), we can write
\[\phi(xy^{-1})=\phi(x)+\sum\limits_{|\alpha|=1}\psi_{\alpha}(x,y)q_{\alpha}(y),\]
where $\psi_{\alpha}\in C^{\infty}(G\times G)$, $q_{\alpha}\in C^{\infty}(G)$, $q_{\alpha}(e)=0$.
Hence
\begin{align*}
\theta(x,\xi)&=\sum\limits_{|\alpha|=1}\int\limits_G\psi_{\alpha}(x,y)q_{\alpha}(y)k_x(y)\xi (y)^*dy\\
&=\sum\limits_{|\alpha|=1}\Delta_{\psi_{\alpha}(x,\cdot)}\Delta_{q_{\alpha}}\sigma(x,\xi).
\end{align*}
Thus 
\beq\|\theta(x,\xi)\|_{op}\leq C\langle\xi\rangle ^{-\frac{na}{2}-(1-a)}.\label{theta1aw}\eq

For the last inequality we have used the following estimate:
\[
\sup_{[\xi]\in\widehat G}\|\Delta_{\psi_{\alpha}(x,\cdot)}\tau(\xi)\|_{op}\leq C\|\psi_{\alpha}(x,\cdot)\|_{C^{k}} \sup_{[\xi]\in\widehat G}\|\tau(\xi)\|_{op},\]
for $k>\frac{n}{2}$. We now write

\[\theta(x,\xi)=\sum\limits_{j=0}^\infty \theta_j(x,\xi),\]  
\noi with $\theta_j(x,\xi)$ supported in $\langle\xi\rangle \sim
2^{j}r^{-1}.$\\

By \eqref{theta1aw} and from the inequalities for the support of $\theta_j$ one has 
 $$\langle\xi\rangle ^{-(1-a)}\leq C2^{-j(1-a)}$$ and 
\[\|\theta_j\|_{l;\mathscr{S}_{1-a,\delta}^{-\frac{na}{2}}}\leq C_l 2^{-j(1-a)}\|\sigma\|_{l;\mathscr{S}_{1-a,\delta}^{-\frac{na}{2}}},\]
 for all $ l$. \\

\indent According to Lemma \ref{l1wa} applied to the symbols $\theta_j$ and $l>\frac{n}{2}$ we have 
\beq\|\left[\phi,\sigma^1(x,D)f\right]\|_{L^\infty}\leq \sum\limits_{j=0}^\infty\|\theta_j(x,D)f\|_{L^\infty}\label{eq:s3}\eq
$$\leq  \sum\limits_{j=0}^\infty
C 2^{-j(1-a)}\|\sigma\|_{l;\mathscr{S}_{1-a,\delta}^{-\frac{na}{2}}} \|f\|_{L^\infty}$$
$$\leq C \|\sigma\|_{l;\mathscr{S}_{1-a,\delta}^{-\frac{na}{2}}} \|f\|_{L^\infty}.$$ 
\noi Since $\phi\geq 1$ on $B(x_0,\delta)$, by using (\ref{eq:s}) and (\ref{eq:s3})
into (\ref{eq:dec}) we have
$$\frac{1}{|B(x_0,\delta)|}\int\limits_{B}|\sigma^1(x,D)f(x)|dx\leq
\frac{1}{|B(x_0,\delta)|}\int\limits_{B}|\phi(x)\cdot\sigma^1(x,D)f(x)|dx$$
$$\leq C\|\sigma\|_{l;\mathscr{S}_{1-a,\delta}^{-\frac{na}{2}}} \|f\|_{L^\infty}, $$
which concludes the proof.
\end{proof}
We now establish a theorem for symbols in $\mathscr{S}_{\rho,\delta}^m(G)$. We recall that the $L^2$ boundedness holds for operators with symbols in $\mathscr{S}_{\rho,\delta}^0(G)$ (cf.   \cite[Proposition 8.1]{vf:psecl}). As a consequence of real interpolation between $L^2$ boundedness and the previous $L^{\infty}-BMO$ boundedness we have:

\begin{thm}\label{t1l2} 
Let $G$ be a compact Lie group of dimension $n$ and let $0<\rho<1$. Let $0\leq\delta<\rho$ and $\sigma\in \mathscr{S}_{\rho,\delta}^{-\frac{n(1-\rho)}{2}}(G)$. Then $\sigma(x,D)$ extends to a bounded operator from $L^{p}(G)$ to $L^{p}(G)$ for $1< p<\infty$. 
\end{thm}

\begin{proof} 
We write $A=\sigma(x,D)$. The symbol $\sigma$ satisfies the condition of Theorem \ref{mth1a}. Hence $A$ is bounded from $L^{\infty}(G)$ to $BMO(G)$.  
 Moreover $A$ is bounded from $L^{2}(G)$ to $L^2(G)$. This implies the boundedness of $A$ from $L^{p}(G)$ to $L^{p}(G)$ for $2\leq p\leq\infty$.
 On the other hand, since $\sigma_{A^*}\in\mathscr{S}_{\rho,\delta}^0(G)$ then $A^*:L^{p}(G)\rightarrow L^{p}(G)$ is bounded for $2\leq p\leq\infty$ and hence by duality we get that $A:L^{p}(G)\rightarrow L^{p}(G)$ is bounded also for $1\leq p\leq 2$.
\end{proof}
\begin{rem}\label{lare1a} The index $\frac{n(1-\rho)}{2}$ in Theorem \ref{t1l2} can not be improved, i.e., if one takes instead an index $\nu_0<\frac{n(1-\rho)}{2}$,
 one only gets $L^p$ boundedness for some finite interval around $p=2$ and not for any $p$ outside that interval. This situation will be explained in more detail by the next theorem and Remark \ref{lare1}.
\end{rem}

We now apply the complex interpolation for an analytic family of operators (cf. \cite{ste:int}, \cite{ste-we:fa}) to obtain $L^p$ bounds for orders $\nu$ with $0\leq\nu\leq\frac{n(1-\rho)}{2}$:

\begin{thm}\label{tci} Let $G$ be a compact Lie group of dimension $n$ and let $0<\rho<1$. Let $0\leq\delta<\rho$ and $\sigma\in \mathscr{S}_{\rho,\delta}^{-\nu}$ with $$0\leq\nu <\frac{n(1-\rho)}{2}.$$ Then $\sigma(x,D)$ extends to a bounded operator from $L^{p}(G)$ to $L^{p}(G)$ for $$\left|\frac{1}{p}-\half\right|\leq \frac{\nu}{n(1-\rho)} .$$ 
\end{thm}
\begin{proof} Let $\sigma\in \mathscr{S}_{1-a,\delta}^{-\nu}(G)$. We consider the family of operators $\{T_z\}_{0\leq \Re\, z\leq 1}$ defined by the symbols
\[\gamma_z(x,\xi):=e^{z^2}\sigma(x,\xi)\langle\xi\rangle^{\nu+\frac{na}{2}(z-1)},\] 
where the operators $T_z$ are defined via the global quantization \eqref{EQ:A-quant}. One can verify that for every $z=t+is$ such that $0\leq t\leq 1$, $s\in\ar$ and $k\in\ene$, we have
\[\|\gamma_z\|_{k,\mathscr{S}_{1-a,\delta}^{0}}\leq e^{t^2-s^2}p(|z|)\|\sigma\|_{k, \mathscr{S}_{1-a,\delta}^{-\nu}},\]
where $p(\lambda)$ is a polynomial of degree $k$.

Since $0\leq t\leq 1, -\infty<s<\infty$ and $e^{s^2}$ dominates $|z|^k$ for $z\in\{0\leq\Re z\leq 1\}$, there exists a constant $C_k>0$ independent of $z$ such that
\[e^{t^2-s^2}p(|z|)\leq C_k.\]
Hence
\beq\label{sem1}\|\gamma_z\|_{k,\mathscr{S}_{1-a,\delta}^{0}}\leq C_k\|\sigma\|_{k, \mathscr{S}_{1-a,\delta}^{-\nu}},\eq
with $C_k>0$ independent of $z$. More precisely, $C_k$ only depends on finitely many semi-norms of $\sigma$.

By the $L^2$ boundedness for $\mathscr{S}_{\rho,\delta}^{0}$ classes there exist a constant $C>0$ and an integer $N$ such that
\beq\label{sem1a}\|T_zf\|_{L^2(G)}\leq C\|\gamma_z\|_{N,\mathscr{S}_{1-a,\delta}^{0}}\|f\|_{L^2(G)}.\eq
From \eqref{sem1} and \eqref{sem1a} we obtain 
 \[\|T_zf\|_{L^2(G)}\leq C_1\|\sigma\|_{N, \mathscr{S}_{1-a,\delta}^{-\nu}}\|f\|_{L^2(G)},\]
for a suitable constant $C_1>0$.

It is clear that the family $\{T_z\}_{0\leq \Re\, z\leq 1}$ is analytic in the strip $$S=\{z=x+iy\in\ce:0<x<1\}$$ and continuous in $\overline{S}$. Thus, the family $\{T_z\}_{0\leq \Re\, z\leq 1}$ defines an analytic family of operators uniformly bounded on $\mathcal{L}(L^2(G), L^2(G))$. In order to apply the complex interpolation we observe that
\[\sup\limits_{-\infty<s<\infty}\|T_{1+is}f\|_{L^2(G)}\leq C_1\|\sigma\|_{N, \mathscr{S}_{1-a,\delta}^{-\nu}}\|f\|_{L^2(G)},\,\,f\in L^2(G),\]
where $C_1$ is independent of $f$.\\

On the other hand 
\[T_{is}f(x)=\int_G\sum\limits_{[\xi]\in \widehat{G}}d_{\xi}\Tr(\xi(y^{-1}x)e^{-s^2}\sigma(x,\xi)\langle\xi\rangle^{\nu-\frac{na}{2}+i\frac{na}{2}s})f(y)dy.\]
Since
\[\gamma_{is}(x,\xi)=e^{-s^2}\sigma(x,\xi)\langle\xi\rangle^{\nu}\langle\xi\rangle^{-\frac{na}{2}}\langle\xi\rangle^{i\frac{na}{2}s},
\]
and 
\[\langle\xi\rangle^{i\frac{na}{2}s}I_{d_{\xi}}\in\mathscr{S}_{1-a,\delta}^{0} \]
(see \cite{Ruzhansky-Wirth:Lp-Z}),
we have
\[\gamma_{is}\in\mathscr{S}_{1-a,\delta}^{-\frac{na}{2}}.\]
Moreover
\[\|\gamma_{is}\|_{k,\mathscr{S}_{1-a,\delta}^{-\frac{na}{2}}}\leq C\|\sigma\|_{k, \mathscr{S}_{1-a,\delta}^{-\nu}},\]
with $C>0$ independent of $s$.

An application of Theorem \ref{mth1a} to the operator $T_{is}$ gives
\[\|T_{is}f\|_{BMO(G)}\leq C\|\sigma\|_{k, \mathscr{S}_{1-a,\delta}^{-\nu}}\|f\|_{L^{\infty}(G)}.\]
The complex interpolation for an analytic family of operators gives us
 \[\|T_{t}f\|_{L^{p}(G)}\leq C_p\|\sigma\|_{k, \mathscr{S}_{1-a,\delta}^{-\nu}}\|f\|_{L^{p}(G)},\]
where $p=\frac 2t,\,0<t\leq 1$. The corresponding symbol of the operator $T_t$ is given by
\[\gamma_{t}(x,\xi)=e^{-t^2}\sigma(x,\xi)\langle\xi\rangle^{\nu+\frac{na}{2}(t-1)}.\]
Since $0\leq\nu\leq\frac{na}{2}$, there exists $t,\,0\leq t\leq 1$ such that
\[\nu=\frac{na}{2}(1-t).\]
Hence the operator $A=\sigma(x,D)$ is bounded from $L^p(G)$ into $L^p(G)$ for $p=\frac 2t$ and
\[\|\sigma(x,D)f\|_{L^p(G)}\leq C_p\|\sigma\|_{k, \mathscr{S}_{1-a,\delta}^{-\nu}}\|f\|_{L^{p}(G)}.\]
We note that 
\[\nu=\frac{na}{2}\p{1-\frac 2p}=na\p{\half-\frac 1p}.\]
By interpolation between $p=2$ and $p=\frac 2t$ we obtain the $L^p(G)$ boundedness for $p$ verifying
\[\half-\frac 1p\leq \frac{\nu}{na}.\]
We can now apply a duality argument for the case $1\leq p\leq 2$. Since the symbol of the operator $A^*=\sigma(x,D)^*$ also belongs to
 $\mathscr{S}_{1-a,\delta}^{-\nu}$ we have $A^*:L^p(G)\rightarrow L^p(G),$
for $\half-\frac 1p\leq \frac{\nu}{na}.$ Then 
\[A:L^p(G)\rightarrow L^p(G)\]
for $\frac 1p-\half\leq \frac{\nu}{na}.$ 
Therefore $A:L^p(G)\rightarrow L^p(G)$ is bounded for
\[\left|\frac 1p-\half\right|\leq \frac{\nu}{na}.\]
\end{proof}

\begin{rem}\label{lare1} 
The index $\frac{n(1-\rho)}{2}$ in Theorem \ref{tci} is sharp. Indeed,  for $G=\mathbb{T}^1$, if $\nu_0 <\frac{1-\rho}{2}$ one only  gets boundedness on finite intervals around $p=2$. This is a consequence of the well-known classical multiplier theory on the torus (cf. \cite{hi:mult}) and Wainger (cf. \cite{wai:trig}). Indeed, let $G=\mathbb{T}^1$ and $0<\rho<1 ,\, 0<\nu_0< \frac{1-\rho}{2}$ and  consider
\[\sigma(\xi)=\frac{e^{i\langle\xi\rangle^a}}{\langle\xi\rangle^{{\nu}_0}},\]
for $\xi\in\zet$.  Then $\sigma\in \mathscr{S}_{\rho,0}^{-\nu_0}(\mathbb{T}^1) $ and the corresponding operator $\sigma(D)$ is bounded on $L^p(\mathbb{T}^1)$ for the interval $$\left|\frac{1}{p}-\half\right|<\frac{{\nu}_0}{a}$$ and is not bounded for $p$ outside that interval. The $L^p$ boundedness inside the interval with centre at $p=2$ can also be obtained from the general result in Theorem \ref{tci}.
\end{rem}

\medskip
 
\noindent{\bf{Acknowledgements}}

\medskip
 
The authors would like to thank the referees 
for several suggestions leading to the improvement of the manuscript.



\end{document}